\documentclass[12pt]{article}

\usepackage{graphicx}
\usepackage{setspace}
\usepackage{amssymb,amsmath}
\usepackage{calc}
\usepackage{verbatim}
\usepackage{epsfig}
\usepackage{graphicx}
\usepackage{graphics}
\usepackage{enumitem}
\usepackage{amsmath}
\usepackage{amsfonts}
\usepackage{amssymb}
\usepackage{verbatim}
\usepackage{caption}
\usepackage{graphicx}
\usepackage{subfigure}
\usepackage{multirow}
\usepackage{array}
\usepackage{wrapfig}
\graphicspath{ {Images/} }
\usepackage{longtable}
\usepackage{afterpage}
\usepackage {tikz}
\usetikzlibrary {positioning}
\usepackage {xcolor}

\begin{document}

\begin{center}
	\begin{Large}{Optimal Domination Polynomials}\end{Large}\\
	
	\vspace{.25in}
	
	 I. Beaton$^{1}$, J.I. Brown$^{1}$ and D. Cox$^{2}$\\
	 
	 	\vspace{.15in}
	 \begin{small}
	  $^{1}$Department of Mathematics and Statistics, Dalhousie University, Halifax, CANADA\\
	 $^{2}$Department of Mathematics,  Mount Saint Vincent University, Halifax, CANADA. Corresponding Author: danielle.cox@msvu.ca\\
		\end{small}

\end{center}

\begin{abstract}
Let $G$ be a graph on $n$ vertices and $m$ edges and $D(G,x)$ the domination polynomial of $G$. In this paper we completely characterize the values of $n$ and $m$ for which optimal graphs exist for domination polynomials. We also show that there does not always exist least optimal graphs for the domination polynomial. Applications to network reliability are highlighted.
\end{abstract}



\newtheorem{theorem}{Theorem}[section]
\newtheorem{lemma}[theorem]{Lemma}
\newtheorem{definition}[theorem]{Definition}
\newtheorem{conjecture}[theorem]{Conjecture}
\newtheorem{proposition}[theorem]{Proposition}
\newtheorem{corollary}[theorem]{Corollary}
\newtheorem{problem}[theorem]{Problem}
\newtheorem{observation}[theorem]{Observation}
\newtheorem{case}{Case}
\newenvironment{proof}{\paragraph{\textbf{Proof}}}
 {\nopagebreak\hfill\nopagebreak\rule{2mm}{2mm}\par\bigskip}
\newcommand*{\QED}{\nopagebreak\hfill\nopagebreak\rule{2mm}{2mm}\par\bigskip}%
\newcommand{\dc}[1]{{\textcolor{blue}{ [\textbf{DC:}}\  \textcolor{blue}{{\em #1}]}}}

\newcommand{\ceil}[1]{\lceil #1 \rceil}
\newcommand{\floor}[1]{\lfloor #1\rfloor}


\section{Introduction}

Consider a graph $G$ with vertex set $V(G)$ and edge set $E(G)$ (we assume throughout that all graphs are {\em simple}, that is, without loops and multiple edges, as neither of these affect domination). Let $S$ be a subset of vertices or edges such that $S$ has a particular graph property, $P$. Perhaps $P$ is that $S$ is independent, complete, a dominating set or a matching. 
The sequences of the number of sets of varying cardinality that have property $P$ have also been studied, particularly through the associated generating polynomials (which are {\em graph polynomials}). Independence, clique, dominating and matching polynomials have all arisen and been studied in this setting.
The evaluation of these polynomials at $1$ yields the counts of the number of subsets in question,  important graph invariants, and all of the graph polynomials can all be considered as functions on the domain $[0,\infty)$. 

If the number of vertices $n$ and edges $m$ are fixed, one can ask whether there exists {\em optimal} graphs with respect to a property, in the following sense. Let $\mathcal{G}_{n,m}$ denote the set of (simple) graphs of order $n$ and size $m$ (that is, with $n$ vertices and $m$ edges). A graph $H\in \mathcal{G}_{n,m}$ is {\em optimal} if $f(H,x) \geq f(G,x)$ for all graphs $G\in \mathcal{G}_{n,m}$ and {\em all} $x \geq 0$ (for any particular value of $x \geq 0$, of course, there is such a graph $H$, as the number of graphs of order $n$ and size $m$ is finite, but we are interested in {\em uniformly} optimal graphs). Of course, if there is a graph $H$ such that the counts for the associated property sets are each greater than or equal to that for any other graph of the same order and size, that graph will be optimal.

Optimality has been studied for independence polynomials \cite{coxbrownindp}, as well for other graph polynomials such as network reliability over the domain $[0,1]$ \cite{ath,suffel,boesch91,browncox,gross,myrvold} and  chromatic polynomials \cite{sakaloglu,chromatic}. 
In this paper we will investigate optimality of domination polynomials. Let $G$ be a graph of order $n$ and size $m$. A subset of vertices is called {\em dominating} if every vertex of $G$ is either in $S$ or adjacent to a vertex of $S$. The cardinality of the smallest dominating set is the domination number, written $\gamma(G)$. We define the {\em domination polynomial} of $G$ as  
\[D(G,x) = \sum_{i=1}^{|V(G)|} d(G,i)x^i = \sum_{i=\gamma (G)}^{|V(G)|} d(G,i)x^i.\] 
where  $d(G,i)$ is the number of dominating sets of cardinality $i$.
We will completely characterize the values of $n$ and $m$ for which optimal graphs exist (which contrasts sharply with the other graph polynomials mentioned where only partial results are known).


\section{Optimality for Domination Polynomials}
\label{sec:Opt}

We begin our study with a useful observation that compares the coefficients of the domination polynomials of two graphs to determine which is more optimal for arbitrarily large and small values of $x$. 

\subsection{Optimal Graphs for Domination Polynomials}
\label{sec:OptGreat}

\begin{observation}
\label{obs:Coeff}
Suppose that $G$ and $H$ are graphs on $n$ vertices and $m$ edges with

$$D(G,x) = \sum_{j=1}^{|V(G)|} d(G,j)x^j$$

\noindent and

$$D(H,x) = \sum_{j=1}^{|V(G)|} d(H,j)x^j$$

\noindent Then

\begin{itemize}
\item if $d(G,j) = d(H,j)$ for $j < l$ but $d(G,l) > d(H,l)$, then $D(G,x) > D(H,x)$ for $x$ arbitrary small positive values of $x$ and

\item if $d(G,j) = d(H,j)$ for $t > j$ but $d(G,t) > d(H,t)$, then $D(G,x) > D(H,x)$ for $x$ arbitrary large.
\end{itemize}
\end{observation}

Given two graphs $G$ and $H$, if $D(G,x)>D(H,x)$ for $x> 0$ then we say that $G$ is more optimal than $H$.

\vspace{0.2in}

Our first result will be regarding the existence of optimal sparse graphs. The following lemma describes an operation that uniformly increases the domination polynomial on $[0,\infty)$.

\begin{lemma}\label{optoperation}
Let $G$ be a graph on $n\geq 3$ vertices with at least one isolated vertex $x$ and at least one edge $e=uv$. Let $H$ be the graph $(G-e) \cup ux$.
Then 
\[ D(H,x)\geq D(G,x) \mbox{ for } x\geq 0. \]
Moreover, if $v$ has degree at least $2$, then 
\[ D(H,x) > D(G,x) \mbox{ for } x > 0. \]
\end{lemma}

\begin{proof}
We begin by showing that every dominating set of $G$ of size $i$ corresponds uniquely to a dominating set of $H$ of the same size.

Let $S_i$ be a dominating set of size $i$ of $G$. Note that since $x$ is an isolated vertex, it appears in every dominating set of $G$.
\begin{itemize}
\item Case 1: If both $u$ and $v$ are in $S_i$ then $S_i$ dominates in $H$.
\item Case 2: If $u\in S_i$, $v\not\in S_i$ then $(S_i-x) \cup \{v\}$ is a dominating set of size $i$ in $H$ which does not dominate in $G$. 
\item Case 3: If $u\notin S_i$, $v\in S_i$ then $S_i$ dominates in $H$ as $x \in S_i$ and $u \in N[x]$. 
\item Case 4: If neither $u$ nor $v$ are an element of $S_i$ both $u$ and $v$ must be dominated in $G-e$, and therefore $S_i$ a dominating set of $H$ as well.
\end{itemize}

Thus, every dominating set of size $i$ of $G$ corresponds to a dominating set of $H$ of size $i$. Moreover, it is not hard to verify that the dominating sets of $H$ produced are different. Hence  $d(H,i)\geq d(G,i)$ for $i\geq 1$ and so $D(H,x)\geq D(G,x)$ for $x\geq 0$ as was to be shown.

Moreover, if $v$ has degree at least $2$, it has another vertex $w \neq u$ adjacent to it. Consider the set $S = V(G) - \{v,x\}$. Then $S$ is not a dominating set of $G$ (as it does not contain $x$ but it is a dominating set in $H$). Moreover, it is straightforward to verify that $S$ is not matched up with any dominating set of $G$ above. It follows that $d(H,n-2) > d(G,n-2)$, and so $D(H,x) > D(G,x)$ for $x > 0.$
\end{proof}

We will now apply this lemma to show the following.

\begin{corollary}\label{optiso}
Let $G$ be a graph on $n\geq 3$ vertices and $m \geq \lceil \frac{n}{2} \rceil$ edges. If $G$ has an isolated vertex, then there exists a graph $H$ of same order and size with no isolated vertices such that $D(H,x)> D(G,x)$ for $x > 0$.
\end{corollary}

\begin{proof}
Let $G'$ be the graph such that $G=G' \cup rK_1$ where $r\geq 1$ is the number of isolated vertices in $G$. Then $G'$ has $n-r$ vertices and $m \geq \lceil \frac{n}{2} \rceil$ edges. We will now show $\Delta(G') \geq 2$. Suppose not -- that is, suppose $\Delta(G') < 2$. Then the sum of all the degrees of vertices in $G'$ is at most $n-r$. Furthermore, 

$$n-r \geq \sum_{v \in G'} \deg (v) = 2m \geq 2\left\lceil \frac{n}{2} \right\rceil\geq n.$$

This is a contradiction as $r\geq 1$. Thus there indeed exists a vertex $v \in G'$ with degree two or more. Let $u \in N(v)$ and $H$ be the graph constructed in Lemma \ref{optoperation} by removing the edge $uv$ from $G$ and adding an edge from $u$ to an isolated vertex. By Lemma \ref{optoperation}, $D(H,x) > D(G,x)$ for $x > 0$ and $H$ has one less isolated vertex. Hence by iterating this process we will find a graph with no isolated vertices which is more optimal than $G$.
\end{proof}

Using the previous result, we can now prove that optimal sparse graphs exist. Two non-isomorphic graphs can have the same domination polynomial, thus for a fixed $n$ and $m$, it is possible for two graphs from $\mathcal{G}_{n,m}$ to both be optimal. If $G \in \mathcal{G}_{n,m}$ is the only optimal graph in $\mathcal{G}_{n,m}$ we call it the \emph{unique optimal} graph.

\begin{corollary}
\label{sparesopt}
For a given $n \geq 2$ and $ m= \lceil \frac{n}{2} \rceil$, the unique optimal graph is $mK_2$ if $n$ is even and $(m-2)K_2 \cup K_{1,2}$ if $n$ is odd.
\end{corollary}

\begin{proof}
Let $G$ be a graph on $n$ vertices and $m= \lceil \frac{n}{2} \rceil$ edges. 
By Corollary~\ref{optiso}, if $G$ has an isolated vertex, there exists a graph $H$ with $n$ vertices, $m$ edges, and no isolated vertices which is more optimal than $G$. Depending on parity of $n$, as $m= \lceil \frac{n}{2} \rceil$ there is only one graph with no isolated vertices: $mK_2$ if $n$ is even and $(m-2)K_2 \cup K_{1,2}$ if $n$ is odd. Hence these graphs must be the unique optimal graphs in their class $\mathcal{G}_{n,m}$.
\end{proof}

\begin{theorem}~\label{sparsmost}
Fix $m\geq 1$ and let $n=2m+r$, $r\geq 0$. Then the unique optimal graph is $m K_2 \cup r K_1$. That is, for $n \geq 2$ and $m < \left\lceil \frac{n}{2} \right\rceil$ a unique optimal graph exists.
\end{theorem}

\begin{proof}
We will induct on $r$. When $r=0$, we know by Corollary~\ref{sparesopt} that $mK_2$ is the unique optimal graph. Suppose that $G_{m,r}=m K_2 \cup r K_1$ is optimal for $1\leq r \leq t$. We will show that for $r=t+1$ that  $m K_2 \cup (t+1)K_1$ is the unique optimal graph.

Let $H$ be a graph on $m$ edges and $n=2m+(t+1)$ vertices. The graph $H$ has at least one isolated vertex. Let $H=H^{\prime}\cup K_1$. Then $D(H,x)=xD(H^{\prime},x)$. Now for $x > 0$, if $H^{\prime} $ is not isomorphic to $G_{m,r-1}$, then for $x > 0$,
\[ D(G_{m,r},x) =xD(G_{m,r-1},x) > xD(H^{\prime},x) = D(H,x),\]
and we are done.
\end{proof}

To contrast, we will now show that optimal graphs need not exist. To do so, we will need the following lemmas regarding the minimum degree of $G$.


\begin{lemma}
\label{lem:MinDeg}
\textnormal{\cite{2010Char}} Let $G$ be a graph of order $n$ then  \[d(G,n-j) = {n \choose j} \mbox{ for all } j \leq \delta (G)\]
\QED
\end{lemma}

\begin{lemma}  \label{Lem:d(G,n-d-1)}
Let $G$ be a graph with $n$ vertices. Then 

$$d(G,n-\delta(G)-1)={n \choose \delta(G)+1} - |\{N[v]:\deg (v)=\delta(G)\}|.$$
\end{lemma}

\begin{proof}
Clearly ${n \choose \delta(G)+1}-d(G,n-\delta(G)-1)$ counts the largest subsets of $V$ which do not dominate $G$. A subset $S \subseteq V$ is a dominating set if and only if for every vertex $v \in V$, $N[v] \cap S \neq \emptyset$. Therefore 
the maximum non-dominating subsets of $V$ are $\{V-N[v]:\deg (v)=\delta(G)\}$. As $|\{V-N[v]:\deg (v)=\delta(G)\}|=|\{N[v]:\deg (v)=\delta(G)\}|$ we get our result.
\end{proof}

\begin{theorem}
\label{thm:0.5nton}
Let $\lceil \frac{n}{2} \rceil <m \leq n-1$. Then for $n\geq 4$ an optimal graph does not exist of order $n$ and size $m$.
\end{theorem}

\begin{proof} To reach a contradiction suppose there exists an optimal graph $G$ with $n$ vertices with $n-r$ edges where $1 \leq r < \lfloor \frac{n}{2} \rfloor$. Consider the domination number of $G$. By Observation \ref{obs:Coeff}, there is no graph with the same order and size of $G$ but of smaller domination number. Let $H=(r-1)K_2 \cup K_{1,n-2r+1}$. As $H$ has $n$ vertices, $n-r$ edges and $\gamma(H)=r$, it follows that $\gamma(G) \leq r$. Furthermore $\gamma(G)$ is bounded below by the number of components in $G$. As $G$ has $n$ vertices and $n-r$ edges, $G$ has at least $r$ components. Therefore $\gamma(G) \geq r$, and so $\gamma(G) = r$. It follows that $G$ must be a disjoint union of $r$ graphs, each with an universal vertex. As $G$ has $n-r$ edges, $G$ must be a forest consisting of $r$ star graphs. 

Again by Observation \ref{obs:Coeff}, there is no graph $F$ with the same order and size of $G$ but with $d(F,r) > d(G,r)$. Let $F=(r-1)K_2 \cup K_{1,n-2r+1}$ and note that $d(F,r)=2^{r-1}$. Thus $d(G,r) \geq 2^{r-1}$. Now $d(G,r)$ is the number of minimum dominating sets in $G$, and thus is equal to the product of the number of minimum dominating sets for each of its $r$ components. However the only star graph with more than one minimum dominating set is $K_2$, which has two. Now $m > \lceil \frac{n}{2} \rceil$ implies $G \not\cong rK_2$, so $G$ has at most $(r-1)$ $K_2$ components. It follows that $n-2r+1 \geq 3$ and $d(G,r) \leq 2^{r-1}$. So $d(G,r) = 2^{r-1}$ and $G \cong F=(r-1)K_2 \cup K_{1,n-2r+1}$ as the last component must also be a star. 

We will now show that a star graph is not optimal, and hence $G$, which has a star component, cannot optimal. Consider $P_n$. By Lemma~\ref{Lem:d(G,n-d-1)}, $d(P_n,n-2)={n \choose 2} - 2$, while  $d(K_{1,n-1},n-2)={n \choose 2} - (n-1)$, and hence $d(P_n,n-2)>d(K_{1,n-1},n-2)$ for $n \geq 4$. Thus by Observation \ref{obs:Coeff}, a star graph is not optimal for $n \geq 4$. This contradiction implies that there is no optimal graph on $n$ vertices and $n-1$ edges for $n \geq 4$.
Thus there cannot exist an optimal graph of order $n$ and size  $\lceil \frac{n}{2} \rceil <m \leq n-1$. 
\end{proof}

Now, we will show that there also does not exist dense graphs that are optimal.

\begin{lemma}
\label{lem:d(g,1)}

\textnormal{\cite{2014Intro}} Let $G$ be a graph of order $n$. Then 
\[d(G,1) = |\{v\in V(G)|deg(v) = n-1\}|. \]
 \QED
\end{lemma}

The {\em join} of two disjoint graphs $G$ and $H$, written $G \vee H$ is the graph formed from there disjoint union by adding in all edges $uv$ where $u$ is a vertex of $G$ and $v$ is a vertex of $H$ (if $G$ and $H$ are not disjoint, one merely uses disjoint isomorphic copies).

\begin{lemma}
\label{lem:graphjoin}

\textnormal{\cite{2014Intro}} Let $G$ be a graph of order $n$. Then
$$D(K_r \vee G,x) = ((1+x)^{r}-1)(1+x)^{n} + D(G_2,x).$$
\QED 
\end{lemma}

\begin{lemma}
\label{lem:OptGreatG}
If a graph $G$ of order $n$ and size $m \geq n-1$ is optimal then $G$ is of the form $K_r \vee H$, the join of $K_r$ and $H$, where $0 \leq r \leq n$ and $H$ is optimal on $n-r$ vertices and at most $n-r-2$ edges.
\end{lemma}

\begin{proof}
By Lemma \ref{lem:d(g,1)} and Observation \ref{obs:Coeff} we wish to maximize the number of degree $n-1$ vertices. Let $r$ be the maximum number of degree $n-1$ vertices $G$ could have with $m$ edges and $n$ vertices. Clearly $0 \leq r \leq n$, $G = K_r \vee H$, and $H$ has $n-r$ vertices. Furthermore $H$ has no degree $n-r-1$ vertices, otherwise such vertices would be degree $n-1$ in $G$. Therefore $H$ has at most $n-r-2$ edges.

Finally we show $H$ is optimal on $n-r$ vertices and $m_H \leq n-r-2$ edges. Let  $H'$ be any another graph of equal order and size to $H$. As $G$ is optimal, $D(G,x)=D(K_r \vee H,x) \geq D(K_r \vee H',x)$ for all $x > 0$. By Lemma \ref{lem:graphjoin}, 

\begin{eqnarray}
D(K_r \vee H,x) & = & ((1+x)^{r}-1)(1+x)^{n-r} + D(H,x) \nonumber \\
D(K_r \vee H',x) & = & ((1+x)^{r}-1)(1+x)^{n-r} + D(H',x) \nonumber
\end{eqnarray}

\noindent Thus $D( H,x) \geq D(H',x)$ for all $x > 0$ and $H$ is optimal.
\end{proof}

\begin{theorem}
For $n \geq 6$ vertices and $n-1 \leq m < {n \choose 2}-6$ there does not exist an optimal graph for the domination polynomial.
\end{theorem}

\begin{proof}
To show a contradiction suppose a graph $G$ of order $n$ and size $m$ is optimal. By Lemma \ref{lem:OptGreatG}, $G$ is the join of $K_r$ and $H$ for some $r \geq 0$ and optimal graph $H$ with $n-r$ vertices and at most $n-r-2$ edges. Let $m_H$ be the number of edges in $H$; then $m = m_H + {{r} \choose {2}} + r(n-r) \geq {{r} \choose {2}} + r(n-r)$. Let $M_G=\{N_G[v]:\deg_G (v)=\delta(G)\}$. Thus it is sufficient to give another graph $G'$, of equal order and size, with $|M_G|>|M_{G'}|$ as Lemma \ref{Lem:d(G,n-d-1)} and Observation \ref{obs:Coeff} imply $D(G',x)>D(G,x)$ for arbitrarily large values of $x$.

We consider the following three cases: $m_H < \ceil{\frac{n-r}{2}}$, $m_H = \ceil{\frac{n-r}{2}}$, and $m_H > \ceil{\frac{n-r}{2}}$. 

\vspace{0.1in}
\noindent \emph{Case 1}:
$m_H < \ceil{\frac{n-r}{2}}$. 

\vspace{0.1in}
\noindent In this case, $H$ is an optimal graph on $n-r$ vertices and less than $\ceil{\frac{n-r}{2}}$ edges. Using Theorem \ref{sparsmost}, $H$ must be the following optimal graph

 $$H=m_HK_2 \cup (n-r-2m_H)K_1.$$

Note that $n-r-2m_H>0$, so $\delta(G)=r$. Furthermore no two vertices of degree $r$ are adjacent. Therefore 

$$|M_G|=|\{v \in V:\deg (v)=r\}|=n-r-2m_H.$$

As $n-1 \leq m < {n \choose 2}-6$ and $m \geq {{r} \choose {2}} + r(n-r)$, it follows that $1 \leq r < n-4$ and hence $|H| = n-r > 4$. Let $u$ be a vertex of minimum degree in $G$, $v$ be any other vertex in $H$, and $x$ be a universal vertex in $G$. Further, let $G'$ be the graph formed by replacing the edge $vx$ in $G$ with the edge $uv$. The graphs $G$ and $G'$ have the same size, order and $\delta(G') \geq \delta(G)$. If $\delta(G') > \delta(G)$, then $d(G',n-r-1)={n \choose r+1}>{n \choose r+1}-(n-r-2m_H)$ and hence we get a contradiction. Thus $\delta(G') = \delta(G)=r$. Every vertex in $G'$, other than $x$ and $u$, has the same degree as they did in $G$. Furthermore $\deg_{G'}(x)=n-2>r$ and $\deg_{G'}(u)=\deg_{G}(u)+1=r+1$. Therefore $G'$ has $n-r-2m_H-1$ vertices of degree $r$. Since $|M_{G'}| \leq |\{v \in V:\deg_{G'} (v)=r\}|$, $|M_G|>|M_{G'}|$.

\vspace{0.1in}

\noindent \emph{Case 2}:
$m_H = \ceil{\frac{n-r}{2}}$.

\vspace{0.1in}
\noindent $H$ has $n-r$ vertices and is optimal. By Corollary \ref{sparesopt}, $H=m_HK_2$ is uniquely optimal if $n-r$ is even and $H=(m_H-2)K_2 \cup K_{1,2}$ is uniquely optimal if $n-r$ is odd. Also $\delta(G) = r+1$, regardless of parity.

\vspace{0.1in}
\hfill\begin{minipage}{\dimexpr\textwidth-1cm}
\emph{Case 2a}: $n-r$ is even.

\vspace{0.1in}
Then $n-r \geq 6$ and without loss of generality let $G = K_r \vee H$ where $H=m_HK_2$ with $m_H \geq 3$. Note that the vertices of degree $r+1$ are exactly the vertices of $H$ and each degree $r+1$ vertex in $H$ shares its closed neighbourhood with its only neighbour in $H$. Therefore $|M_G|=m_H$.

\vspace{0.1in}

Let $u_1,u_2,v_1, v_2$ and $x$ be vertices in $G$ such that $x$ is a universal vertex in $G$ and $u_1,u_2$ and $v_1, v_2$ each induce $K_2$ components in $H$. Note $N_G[u_1] =N_G[u_2] \in M_G$ and $N_G[v_1] =N_G[v_2] \in M_G$. Let $G'$ be the graph formed by replacing the edges $xu_1$, $xu_2$, $xv_1$ and $xv_2$ with $v_1u_1$, $v_1u_2$, $v_2u_1$ and $v_2u_2$. Note the degree of $u_1,u_2,v_1$ and  $v_2$ have all increased from $G$ to $G'$, $\deg_{G'}(x)=\deg_{G}(x)-4=n-5 \geq r+1$, and the closed neighbourhood of every other vertex is unchanged. Therefore $M_{G'} \subseteq M_{G} \cup \{N_{G'}[x]\} -\{N_G[u_1],N_G[v_1]\}$ and hence $|M_G|>|M_{G'}|$.

\xdef\tpd{\the\prevdepth}
\end{minipage}

\vspace{0.1in}

\vspace{0.1in}
\hfill\begin{minipage}{\dimexpr\textwidth-1cm}
\emph{Case 2b}: $n-r$ is odd. 

\vspace{0.1in}

Then $n-r \geq 5$ and without loss of generality let $G = K_r \vee H$ where $H=(m_H-2)K_2 \cup K_{1,2}$ with $m_H-2 \geq 1$. Let $u_1,u_2,v$ and $x$ be vertices in $G$ such that $x$ is a universal vertex in $G$, $u_1,u_2$ induce a $K_2$ component in $H$ and $v$ is a leaf in the $K_{1,2}$ component of $H$. Note $N_G[u_1] =N_G[u_2] \in M_G$ and $N_G[v] \in M_G$. Let $G'$ be the graph formed by replacing the edges $xu_1$ and $xu_2$ with $vu_1$, $vu_2$. The degree of $u_1$ and $u_2$ remain $r+1$ and $N_{G'}[u_1] =N_{G'}[u_2] \in M_{G'}$. Furthermore $\deg_{G'}(x)=\deg_{G}(x)-2=n-3 > r+1$ and $\deg_{G'}(v)=\deg_{G}(v)+2=r+3$ so $N_{G'}[v] \notin M_{G'}$. As the closed neighbourhood of every other vertex is unchanged, $|M_G|>|M_{G'}|$.

\xdef\tpd{\the\prevdepth}
\end{minipage}

\vspace{0.1in}

\noindent \emph{Case 3}:
$m_H > \ceil{\frac{n-r}{2}}$. 

\vspace{0.1in}
By Lemma \ref{lem:OptGreatG}, $H$ is an optimal graph on $n-r$ vertices and $m_H > \ceil{\frac{n-r}{2}}$ edges, where $m_H \leq n-r-2 < n-r-1$. As in case 1, $n-r \geq 5$. By Theorem \ref{thm:0.5nton}, there is no optimal graph on $n-r$ vertices and $m_H > \ceil{\frac{n-r}{2}}$ edges. Thus this case is a contradiction. 
\end{proof}

Clearly for $m={n \choose 2}$ and $m={n\choose 2}-1$, unique optimal graphs exist, since there is only one graph in each case, but we now show for other dense graphs optimal graphs do not exist.

\begin{theorem}
\label{optimalk}
Let $G$ be a graph on $n\geq 6$ vertices and $m={n \choose 2}-k$, $2\leq k\leq 6$ Then an optimal graph does not exist.
\end{theorem}

\begin{proof}
By Observation~\ref{obs:Coeff} we know that ane optimal graph for values of $x$ close to 0 will have the most number of universal vertices. For $k=2,3,4,5,6$ we will show that the graph $H_k$ which is $K_n$ with a matching of size $k$ removed is optimal for larger values of $x$. It is easy to see that $D(H_k,x)=(1+x)^n-1-2kx$. Note that $d(H_k,i)={n \choose i}$ for $i\geq 2$.

\begin{itemize}
\item 
For $k=2$ consider the graph $G$ which is $K_n$ with the edges of a $P_3$ removed. The domination polynomial for this graph is $D(G,x)=(1+x)^n-1-(x^2+3x)$. This is the unique graph of order $n$ and $m={n \choose 2}-2$ with $n-3$ universal vertices.

\item For $k=3$ let $G$ be a $K_n$ with the edges of a $K_3$ removed, which has a domination polynomial of $D(G,x)=(1+x)^n-1-(3x^2+3x)$. The graph $G$ is the unique graph with $n-3$ universal vertices. 

\item For $k=4$ there are two graphs of order $n$ and size $m = m={n \choose 2}-k$ on with $n-4$ universal vertices, $G$, namely $K_n$ with the edges of a $C_4$ removed and $K_n$ with the edges of a $K_3$ with a leaf removed. We can compute that $D(G,x)=(1+x)^n-1-(2x^2+4x)$ and $D(G',x)=(1+x)^n-1-(x^3+3x^2+4x)$. It is easy to see that $D(G,x)\geq D(G',x)$ for $x \geq 0$. 

\item For $k=5$ consider the graph $G$ which is $K_n$ with the edges of a $K_4$ with an edge removed. The domination polynomial for this graph is $D(G,x)=(1+x)^n-1-(2x^3+6x^2+4x)$. This is the unique graph of order $n$ and $m={n \choose 2}-5$ with $n-4$ universal vertices. 

\item Lastly, for $k=6$ consider the graph $G$ which is $K_n$ with the edges of a $K_4$ removed. The domination polynomial for this graph is $D(G,x)=(1+x)^n-1-(4x^3+6x^2+4x)$. This is the unique graph of order $n$ and size $m={n \choose 2}-6$ with $n-4$ universal vertices. 
\end{itemize}

Clearly for all the above described graphs $d(G,i)= {n\choose i}$ for $i=4,\ldots n$, but $d(G,i) < {n \choose 2}$ for some $i \in \{2,3\}$. For the largest value of $i$ where $d(G,i)<{n \choose i}$ it is the case that $d(H_k,i)={n \choose i}$, therefore by Observation~\ref{obs:Coeff} $H_k$ is optimal for arbitrarily large values of $x$, hence an optimal graph does not exist.
\end{proof}


\begin{corollary}
\label{cor:compile}
For graphs of order $n\geq 6$,

\begin{itemize}
\item $mK_2 \cup rK_1$ where $r = n-2m$ is uniquely optimal when $m < \lceil \frac{n}{2} \rceil$.
\item $mK_2$ is uniquely optimal when $n$ is even and $m = \lceil \frac{n}{2} \rceil$.
\item $(m-2)K_2 \cup K_{1,2}$ is uniquely optimal when $n$ is odd and $m = \lceil \frac{n}{2} \rceil$.
\item No optimal graph exists for $\lceil \frac{n}{2} \rceil <m  < {n \choose 2}-1$.
\item $K_n-e$ is uniquely optimal for $m = {n \choose 2}-1$, for $e\in E(G)$.
\item $K_n$ is uniquely optimal for $m = {n \choose 2}$.
\end{itemize}

\QED
\end{corollary}
In fact, via some calculations, Corollary \ref{cor:compile} can been seen to hold for $n < 6$ as well, with the exception of $K_1 \vee 2K_2$ which is the unique optimal graph on five vertices and six edges.

\section{Conclusion}
\label{sec:Conclusion}

In \cite{domrel} the domination reliability polynomial was defined as follows. For a given graph $G$ we assume that vertices are independently operational with probability $p \in [0,1]$; the {\em domination reliability} $Drel(G,p)$ of $G$ is the probability that the operational vertices form a  dominating set of the graph. As for all-terminal reliability, the existence of optimal reliability polynomials is an open area of study. Noting that $Drel(G,p) = (1-p)^{n}\cdot D(G,\frac{p}{1-p})$, from Corollary~\ref{cor:compile} we obtain a complete characterization of values of $n$ and $m$ for which optimal graphs exist for domination reliability.

\begin{corollary}\label{domrel}
For $n\geq 6$ and $m\leq \lceil \frac{n}{2}\rceil$ uniquely optimal graphs exist for domination reliability. For  $\lceil \frac{n}{2} \rceil <m  < {n \choose 2}-1$ optimal graphs do not exist for domination reliability. For  ${n \choose 2}-1 \geq m  \geq {n \choose 2}$ uniquely optimal graphs exist for domination reliability.
\QED
\end{corollary}

On another note, we can ask what graphs are the {\em least}-optimal (a graph $H\in \mathcal{G}_{n,m}$ is {\em least-optimal} if $f(H,x) \leq f(G,x)$ for all graphs $G\in \mathcal{G}_{n,m}$ and {\em all} $x \geq 0$). While, of course, for $m = n-1$ or $n$ there are least-optimal graphs (as there is only a single graph in each such class), we can show that, in general, such graphs need not exist.

\begin{theorem}\label{thm:leastopt}
Let $G$ be a graph on $n\geq 7$ vertices and $m={n \choose 2}-k$, $2 \leq k \leq \frac{n}{2}$ edges. Then a least-optimal graph does not exist.
\end{theorem}

\begin{proof}
By Observation~\ref{obs:Coeff} we know that the least optimal graph for values of $x$ close to 0 will have the least number of universal vertices. For $k=2$ there are only two possible graphs and so by Theorem~\ref{optimalk} an least optimal graph does not exist.
Thus we can assume $k \geq 3$.

Let $G_k$ be $K_n$ with the edges of a matching of size $k$ removed. This is the unique graph of order $n$ and size $m={n \choose 2}-k$, $2 \leq k \leq \frac{n}{2}$ that has $n-2k$ non-universal vertices, that is, it is least optimal for values of $x$ near 0. The domination polynomial for this graph is $D(G_k,x)=(1+x)^{n}-1-2kx$. Consider the graph $H$, which is $K_n$ with the edges of a $P_{k+1}$ removed. The domination polynomial for this graph is $D(H,x)=(1+x)^n-1-(k+1)x-(k-2)x^2$. Since $D(G_k,x) \leq D(H,x)$ holds if and only if $(-k+1)x+(k-2)x^2\leq 0$, which is true precisely when $x \leq \frac{k-1}{k-2}$, so outside this range (i.e. when $x \in (k-1)/(k-2),\infty)$), $H$ is less optimal. Thus a least-optimal graph does not exist.
\end{proof}

An open problem is to characterize the values of $n$ and $m$ such that least optimal graphs exist.

\vspace{0.5cm}

\noindent {\bf Acknowledgements}\\

\vspace{0.05in}
\noindent J.I. Brown acknowledges support from NSERC (grant application RGPIN 170450- 2013). D. Cox acknowledges research support from NSERC (grant application RGPIN 2017-04401).


\vspace{0.1in}


\begin{thebibliography}{}

\bibitem{2010Char}
S. Akbari, S. Alikhani, and Y.H. Peng, Characterization of graphs using domination polynomials, Eur. J. Comb., 31, 1714--1724, (2010).

\bibitem{2014Intro}
S. Alikhani and Y. Peng, Introduction to domination polynomial of a graph, Ars.Combin., 114, 257--266, (2014).

\bibitem{ath}
Y. Ath and M. Sobel, Some conjectured uniformly optimal reliable
networks, Probab.\ Engrg.\ Inform.\ Sci., 14,375--383 (2011).


\bibitem{suffel}
F.T. Boesch and A. Satyanarayana and C.L. Suffel, Least reliable networks and the reliability domination, IEEE Trans. Comm., 38, 2004--2009 (1990).
  
\bibitem{boesch91}
F. Boesch and X. Li and C. Suffel, On the existence of uniformly optimally reliable networks, Networks, 21, 181--194 (1991).


\bibitem{browncox}
J.I. Brown and D. Cox, Nonexistence of optimal graphs for all terminal reliability, Networks, 63, 146--153 (2014).

\bibitem{coxbrownindp}
J. I. Brown and D. Cox, Optimal Graphs for Independence and $k$-Independence polynomials, Graphs and Combinatorics, 34(6), 1445--157, 2018.

\bibitem{domrel}
Klaus Dohmen and Peter Tittmann, Domination Reliability, Electr. J. Comb., 19(1), 1--14 (2012).

\bibitem{gross}
D. Gross and J.T. Saccoman, Uniformly optimal reliable graphs, Networks, 31, 217--225 (1998).
 

\bibitem{myrvold}
W. Myrvold and K.H. Cheung and L.B. Page and J.E. Perry, Uniformly most reliable networks do not always exist, Networks, 21, 417--419 (1991).

\bibitem{sakaloglu}
A. Sakaloglu and A. Satyanarayana, Graphs with the least number of colourings, J. Graph Theory, 19, 523--533 (1995).

\bibitem{chromatic}
I. Simonelli, Optimal graphs for chromatic polynomials, Discrete Mathematics 308(11), 2228-2239 (2008).


\end{thebibliography}


\end{document}